\documentclass[11pt,a4paper]{article}

\usepackage[utf8]{inputenc}
\usepackage[T1]{fontenc}
\usepackage{lmodern}
\usepackage[a4paper,margin=2.5cm]{geometry}
\usepackage{amsmath,amsfonts,amssymb,bm}
\usepackage{amsthm}
\usepackage{bbm}
\usepackage{graphicx}
\usepackage{xcolor}
\usepackage{url}
\usepackage[authoryear,round]{natbib}
\setcitestyle{authoryear,round,citesep={;},aysep={,},yysep={;}}
\usepackage[colorlinks=true,linkcolor=blue!60!black,citecolor=blue!60!black,urlcolor=blue!60!black]{hyperref}
\hypersetup{
  pdftitle={On Sparsity and Sub-Gaussianity in the Johnson-Lindenstrauss Lemma},
  pdfauthor={Aurélien Garivier, Emmanuel Pilliat},
  pdfkeywords={Johnson-Lindenstrauss lemma, random projections, sparsity, sub-Gaussian variables, dimensionality reduction}
}
\newcommand{\R}{\mathbb{R}}
\renewcommand{\P}{\mathbb{P}}
\newcommand{\E}{\mathbb{E}}
\newcommand{\Var}{\mathbb{V}\!\mathrm{ar}}
\DeclareMathOperator{\Tr}{Tr}
\newcommand{\Diag}{\mathrm{Diag}}
\def\1{\mathbbm{1}}

\newtheorem{theorem}{Theorem}
\newtheorem{proposition}{Proposition}

\title{On Sparsity and Sub-Gaussianity in the Johnson-Lindenstrauss Lemma\thanks{Published in \emph{Transactions on Machine Learning Research} (10/2025). Reviewed on OpenReview: \url{https://openreview.net/forum?id=Znaty8V3a3}. This version corrects a few typos and constants with respect to the published version.}}
\author{Aurélien Garivier\\
  \small ENS de Lyon, UMPA UMR 5669 / LIP UMR 5668\\
  \small 46 allée d'Italie, F-69364 Lyon cedex 07, France\\
  \small \texttt{aurelien.garivier@ens-lyon.fr}
  \and
  Emmanuel Pilliat\\
  \small ENSAI Rennes, France\\
  \small \texttt{emmanuel.pilliat@ensai.fr}
}
\date{}

\begin{document}

\maketitle

\begin{abstract}
The Johnson-Lindenstrauss lemma is the theoretical core of the dimensionality reduction methods based on random projections. While a standard probabilistic proof uses matrices with Gaussian entries, the computational cost of random projections can be drastically reduced by the use of simpler variables, especially if they vanish with a high probability. In this paper, we propose an elementary analysis of random projections under classical assumptions that emphasizes the key role of sub-Gaussianity: any $1$-sub-Gaussian distribution with variance $1$ yields the same sufficient target dimension as the Gaussian distribution. We then extend the analysis to sparse projections, with data-dependent sparsity conditions on the pairwise differences of the points, and we quantify the cost of sparsity on the target dimension. A complementary lower bound shows the limits induced by the sparsity of the data itself.
\end{abstract}

\section{Introduction}
The celebrated Johnson-Lindenstrauss lemma~\cite{JL84}  ensures the existence of low-distortion embeddings of points from high-dimensional into low-dimensional Euclidean space. If $x_1,\dots,x_n\in\R^p$ (with $n\geq 2$), where $p$ is a (large) integer, and if $\epsilon\in(0,1)$ is a tolerance parameter, then there exists a matrix $A$ in the set $\mathcal{M}_{d,p}(\R)$ of real matrices with $d$ rows and $p$ columns such that 
\begin{equation}\label{eq:JLprop} \forall 1\leq i,j\leq n, \quad (1-\epsilon) \|x_i-x_j\|^2\leq \|Ax_i-Ax_j\|^2 \leq (1+\epsilon) \|x_i-x_j\|^2 \end{equation}
as soon as 
\begin{equation}\label{eq:JLcond}
	d\geq \frac{8\ln(n)}{\epsilon^2-\epsilon^3}\;.\end{equation}
The classical proof of this result is an illustration of the Probabilistic Method~\cite{alon04}: when drawing the entries of $A$ at random from independent Gaussian distributions, Property~(\ref{eq:JLprop}) is satisfied with positive probability as soon as the output space is large enough. It results from a simple deviation bound for the chi-square distribution, and hence builds on the specificity of the Gaussian distribution. This proof also gives mathematical foundations for \emph{random projections}, a simple and computationally efficient dimensionality reduction technique in unsupervised machine learning (see e.g.~\cite{randomProj01,indyk01,vempala04,sarlos06,clarkson09} and references therein). 

\cite{Achlioptas01} showed that random projections can easily be extended to non-Gaussian matrices. In particular, Rademacher, or $\{-1, 0, 1\}$-valued entries  can just as well be chosen, leading to even simpler algorithms suitable for database applications. The proof provided in \cite{Achlioptas01} relies on moment bounds and is somewhat specific to those two families of distributions. It was improved by Matou\v{s}ek in~\cite{matouvsek08} with more general arguments. It is generally considered~\cite{verySparse06} that "a uniform distribution is easier to generate than normals, but the analysis is more difficult". Even faster methods for sparse data on streams were then devised \cite{NirChazelle09,CHARIKAR20043,sparserJL14} using random hashing constructions and more involved moment bounds. Very recently and concurrently to our work, \cite{simpleJL24} has proposed a unified analysis of sparse Johnson-Lindenstrauss methods based on the Hanson-Wright inequality, while \cite{pmlr-v235-hogsgaard24a} studies the optimal sparsity of the embedding matrix (number of non-zero entries per column) as a function of the dimension $d$, the number of points $n$ and the tolerance parameter $\epsilon$.

Our goal is to make the analysis of sparse random projections accessible to a wider audience, and the main contribution of this paper is two-fold. 
The first purpose is to highlight that \emph{sub-Gaussianity} is in fact an elementary property of random matrix entries that suffices to ensure the success of random projections. The possibility of using sub-Gaussian entries was (to the best of our knowledge) first mentioned in~\cite{matouvsek08}. 
Unlike \cite{simpleJL24}, our analysis does not require the Hanson-Wright inequality: it relies on a Gaussian comparison of exponential moments, and on Chernoff and Bernstein bounds.
Compared to~\cite{matouvsek08,chafai:compressedSensing}, the central argument is shorter, and we pay special attention to the constant in the exponential deviation bound.  
The folk theorem (stated for example in~\cite{boucheron2013concentration}) is that using sub-Gaussian entries instead of normal laws is possible at the price of "just a multiplicative constant" in the target dimension. Focusing on this multiplicative constant, we show what distributions can be used to obtain \emph{exactly the same guarantees} as with Gaussian entries. Trying to provide an optimal compromise between tightness and clarity, we also give intuitions and intermediate products that may be overlooked with the general point of view of Orlicz norms. We prove that any $1$-sub-Gaussian law with variance $1$ offers the same guarantees as the Gaussian law. We provide a short analysis of Achlioptas' limit of $1/3$ of non-zero coefficients, and under what conditions it is possible to go further in the understanding of much sparser random projections. 
Interestingly, our treatment of the lower and upper bound of (\ref{eq:JLprop}) is not totally symmetric. While the upper deviations of sub-Gaussian variables can be handled by Chernoff's bound just as those of the Gaussian law, the lower deviations can obviously be much smaller (after all, the constant $0$ is $1$-sub-Gaussian) and hence require a different argument.
To highlight the connection to the Hanson-Wright inequality, we propose in the final section an elementary proof of the upper-tail part of this inequality for sub-Gaussian vectors when the quadratic form is positive semidefinite. It relies on the same tools as the other proofs of this paper, and it can be read independently. 

The second purpose of this paper is to build on this analysis to clearly emphasize the conditions on the data under which much sparser random projection matrices can be considered. Our approach relies on a decomposition of the entries of the projection matrix into a sparse component $\zeta$ and a sub-Gaussian component $U$. The take-home message is that the distances are preserved as soon as the proportion of non-zero entries in the projection matrix $A$ multiplied by the effective number of non-zero coefficients in each pairwise difference $x_i - x_j$ is sufficiently large: we give sufficient sparsity conditions in terms of the normalized pairwise differences, and a complementary lower bound for flat sparse vectors under Bernoulli-Rademacher projections. Together, these bounds identify the scale $1/s$, up to constants and logarithmic factors, in the stated regime. An analysis of the same flavor was given by \cite{matouvsek08}, as a high-probability guarantee for a fixed vector and a specific distribution of the sparse projection $A$; passing to a finite set of pairwise differences is then a standard union bound. Our emphasis is on an elementary treatment, with detailed proofs in the sparse and dense cases and explicit constants in the dense case. Furthermore, we fully elaborate  on the lower bound argument outlined in that article. Specifically, we establish that, for the Bernoulli-Rademacher ensemble and flat $s$-sparse vectors, the condition $q\gtrsim 1/s$ on the proportion of non-zero coefficients of $A$ cannot be improved by more than constants and logarithmic factors when $d\asymp 1/\epsilon^2$.

The paper is organized as follows. The pedagogical Section~\ref{sec:data-agnostic} includes our analysis of random projections without assumptions on the data: Section~\ref{sec:chernoff} includes a deviation bound for the averages of squared sub-Gaussian variables which is applied in Section~\ref{sec:JL} to derive the classical Johnson-Lindenstrauss lemma for sub-Gaussian random matrices. We discuss in Section~\ref{sec:RP} a few examples of choices of the distribution $P$ for random projections. 
Section~\ref{sec:sparseRP} investigates the possibility of much sparser projection matrices and of the theoretical limit to the minimal sparsity. Theorem \ref{thm:JLs1s}, with its short proof in Section~\ref{sec:proofs}, extends the previous analysis with minimal changes to sparse matrices. Theorem \ref{th:jl_sparse} gives the order of magnitude of the minimal allowed sparsity to obtain a quasi-isometry with high probability, at the price of poly-logarithmic terms. The optimality of this result is discussed in Section~\ref{subsec:optimality}. A connection to the Hanson-Wright inequality is proposed in Section~\ref{sec:HW}, before the full proofs of the main theorems in the appendix (Section~\ref{sec:proofs}).

\section{Data-agnostic random projections}\label{sec:data-agnostic}
We recall in this section known but fundamental results that are of constant use in the sequel. The Johnson-Lindenstrauss lemma is stated from the start for sub-Gaussian variables, and the comparison (\ref{eq:gaussianTrick}) of exponential moments with the Gaussian case is used as the central tool. Section~\ref{sec:RP} contains a short derivation of results published in~\cite{Achlioptas01}, with a discussion on their optimality.
\subsection{Chernoff's method for squared sub-Gaussian variables}\label{sec:chernoff}
Let $X$ be a random variable  assumed to be $1$-sub-Gaussian, which means that $\forall \lambda\in\R, \E\big[e^{\lambda X}\big]\leq e^{\lambda^2/2}$. 
This implies in particular that $\E[X] = 0$ and that $\Var[X]\leq 1$. We derive in this section a deviation bound for the empirical mean of independent copies of $X^2$:
\begin{proposition}\label{prop:devX2}
If $X_1,\dots,X_d$ are iid $1$-sub-Gaussian random variables with variance $1$, then for every $\epsilon\in(0,1)$
\[\P\left(\frac{X_1^2+\dots+X_d^2}{d}\notin \big[1-\epsilon, 1+\epsilon\big]\right)\leq 2 e^{-d\left(\frac{\epsilon^2 - \epsilon^3}{4} \right)} 
\;.\]
\end{proposition}
The restriction to $\epsilon<1$ is harmless: for $\epsilon\geq 1$, the right-hand side is at least $2$ and the statement is trivial.

For Gaussian variables, this is a well-known application of Chernoff's method that is to be found in many probability textbooks. Inspired in particular by Theorem 2.6 of~\cite{wainwright_2019}, we propose an extension to sub-Gaussian variables based on a Gaussian comparison of the exponential moments of $X^2$ (a special case of the quadratic-form comparison of \cite{hsu2012tail}, see Section~\ref{sec:HW}). The proof requires to treat the upper- and the lower bound separately, which is done in the two following subsections.

\subsubsection{Proof of the upper bound}\label{subsec:UB}
Chernoff's method requires to bound the exponential moments $\E\left[e^{\ell X^2}\right]$ of $X^2$ with $\ell>0$ for the right deviations and with $\ell<0$ for the left deviations. 
We start with the right deviations, for which we will see right away that a reduction to the Gaussian case is possible without further assumption. Following \cite{wainwright_2019} (Theorem 2.6), and remarking that for all $x\in\R$, and $\ell>0$,  \[e^{\ell x^2} = \int_{-\infty}^{\infty} e^{\lambda x}  \;\frac{e^{-\frac{\lambda^2}{4\ell}}}{2\sqrt{\pi\ell}} \;d\lambda\;,\]
if $X$ is $1$-sub-Gaussian we obtain by Fubini's theorem that for  every $\ell\in(0,1/2)$
\begin{align*}
	\E\left[e^{\ell X^2}\right]&=\E\left[ \int_{-\infty}^{\infty}  e^{\lambda X}  \;\frac{e^{-\frac{\lambda^2}{4\ell}}}{2\sqrt{\pi\ell}} \;d\lambda\right]	=\int_{-\infty}^{\infty}\E\left[  e^{\lambda X} \right]\;\frac{e^{-\frac{\lambda^2}{4\ell}}}{2\sqrt{\pi\ell}} \;d\lambda \\
& \leq  \int_{-\infty}^{\infty} e^{\frac{\lambda^2}{2}} \;\frac{e^{-\frac{\lambda^2}{4\ell}}}{2\sqrt{\pi\ell}} \;d\lambda
= \int_{-\infty}^{\infty}e^{-\frac{\lambda^2(1-2\ell)}{4\ell}} \frac{d\lambda}{2\sqrt{\pi\ell}} = \frac{1}{\sqrt{1-2\ell}}\;,
\end{align*}
which holds with equality if and only if $X\sim\mathcal{N}(0,1)$.
Another way of establishing the above bound is to observe that for $G\sim\mathcal{N}(0,1)$, Fubini's theorem implies that
\begin{align}
	\E_X\left[e^{\ell X^2}\right] & = \E_X\left[ \E_G\left[ e^{\sqrt{2\ell}X \,G}\right] \right] =  \E_G\left[ \E_X\left[ e^{\sqrt{2\ell}G \,X}\right] \right]  \nonumber\\
	&\leq \E_G\left[ e^{\ell G^2 }\right] = \frac{1}{\sqrt{2\pi}} \int_\R e^{\ell u^2} e^{-\frac{u^2}{2}}du 
 = \frac{1}{\sqrt{1-2\ell}} \label{eq:gaussianTrick}
\end{align}
with equality if and only if $X\sim\mathcal{N}(0,1)$.

Hence, the squares of all $1$-sub-Gaussian variables have exponential moments bounded by those of the square of a standard Gaussian variable, which permits the right-deviations to be handled the usual way. 
If $Z_1,\dots,Z_d$ are independent random variables with the same distribution as $X^2$, then  for every positive $\epsilon$, Markov's inequality implies that 
\[\P\left(\frac{Z_1+\dots+Z_d}{d}\geq 1+\epsilon\right) = \P\left(e^{\ell\big(Z_1+\dots+Z_d\big)}\geq e^{d\ell(1+\epsilon)}\right) \leq \frac{\E[e^{\ell Z_1}]^d}{e^{d\ell(1+\epsilon)}} =  e^{-d \left(\ell (1+\epsilon) - \ln \E\left[e^{\ell Z_1}\right]\right)}\;.\]
The concave function $\ell\mapsto \ell (1+\epsilon) + \frac{1}{2}\ln(1-2\ell) \leq \ell (1+\epsilon) -\ln \E\big[e^{\ell Z_1}\big]$ is maximized at $\ell^*$ such that $1+\epsilon=\frac{1}{1-2\ell^*}$, that is at $\ell^*=\frac{1}{2}\left(1-\frac{1}{1+\epsilon}\right)=\frac{\epsilon}{2(1+\epsilon)}$. Hence, 
\[\P\big(Z_1+\dots+Z_d \geq (1+\epsilon)d\big) \leq \exp\Big(-d\big(\ell^*(1+\epsilon)+\tfrac{1}{2}\ln(1-2\ell^*)\big)\Big) = \exp\left(-d\,\frac{\epsilon-\ln(1+\epsilon)}{2}\right)\;. \] 
This expression can be slightly simplified in many different ways. Let us illustrate the very useful "Pollard trick": taking $g(\epsilon)=\epsilon-\ln(1+\epsilon)$, since $g(0)=g'(0)=0$  and since $g''(\epsilon)=1/(1+\epsilon)^2$ is convex, by Jensen's inequality
\[\frac{\epsilon-\ln(1+\epsilon)}{\epsilon^2/2} = \int_{0}^1 g''(s\epsilon) 2(1-s)ds \geq g''\left(\epsilon\int_{0}^1 s\; 2(1-s)ds\right) = g''\left(\frac{\epsilon}{3}\right)\;,\]
and hence $\displaystyle{\frac{\epsilon-\ln(1+\epsilon)}{2} \geq \frac{\epsilon^2}{4\left(1+\frac{\epsilon}{3}\right)^2}} \geq \frac{\epsilon^2 - \epsilon^3}{4} $.
In summary, 
\begin{equation}\label{eq:chernoffUB}\P\left(\frac{Z_1+\dots+Z_d}{d}\geq 1+\epsilon\right)\leq 
	e^{-d\left(\frac{\epsilon^2 - \epsilon^3}{4} \right)} 
	\;.\end{equation}

\subsubsection{Proof of the lower bound}\label{subsec:LB}
There is no hope to prove  that $\E\left[e^{-\ell X^2}\right] \leq \frac{1}{\sqrt{1+2\ell}}$  for all $\ell>0$ and all $1$-sub-Gaussian distributions, since it is for example not the case if $X=0$ almost surely.
In the context of the Johnson-Lindenstrauss lemma, it is very natural to assume that the entries of the random matrix have variance 1, so that at least $\E\big[\|Ax_i-Ax_j\|^2\big] = \|x_i-x_j\|^2 $.
But even under this assumption, the negative exponential moments of $X^2$ are not bounded by those of the standard Gaussian: for the three-point law $P = \frac{1}{6}\delta_{-\sqrt{3}} + \frac{2}{3}\delta_0 + \frac{1}{6}\delta_{\sqrt{3}}$, which has variance $1$ and is $1$-sub-Gaussian (see Proposition~\ref{prop:ArchiBest} below), $\E\big[e^{-\ell X^2}\big] = \frac{2}{3} + \frac{1}{3}e^{-3\ell}$, which exceeds $1/\sqrt{1+2\ell}$ for $\ell = 1$ ($0.683$ versus $0.577$) and tends to $2/3$ instead of $0$ as $\ell\to\infty$. The left deviations hence require a different argument.

We remark that if $\Var[X] = 1$, the sub-Gaussianity inequality $\E\big[e^{\lambda X}\big]\leq e^{\lambda^2/2}$, averaged at $\lambda$ and $-\lambda$, gives $\E[\cosh(\lambda X)] \leq e^{\lambda^2/2}$; expanding both sides around $\lambda=0$ (all moments of $X$ are finite), $1 + \frac{\lambda^2}{2}\E[X^2] + \frac{\lambda^4}{24}\E[X^4] + O(\lambda^6) \leq 1 + \frac{\lambda^2}{2} + \frac{\lambda^4}{8} + O(\lambda^6)$, and since $\E[X^2]=1$, the comparison of the fourth-order coefficients yields $\E[X^4] \leq 3$.
Using that $e^{-u} \leq 1-u+\frac{u^2}{2}$ for all $u\geq 0$, we obtain that 
$\E\left[e^{- \ell X^2}\right] \leq 1 - \ell + \frac{3\ell^2}{2}$ and hence
\[\P\left(\frac{Z_1+\dots+Z_d}{d}\leq 1-\epsilon\right) \leq  e^{-d\left( \ell(-1+\epsilon) -\ln\left(1-\ell+\frac{3\ell^2}{2}\right)\right)} 
\;.\]
Since $-\ln(1-u)\geq u+u^2/2$ for all $u\in[0,1)$, and since $u=\ell-3\ell^2/2\in[0,1)$ for $\ell=\epsilon/2\in(0,1/2)$,
\[\ell(-1+\epsilon ) -\ln\left(1-\ell+\frac{3\ell^2}{2}\right) \geq \ell\epsilon - \frac{3\ell^2}{2}  + \frac{(\ell-3\ell^2/2)^2}{2}  \geq \ell\epsilon - \ell^2 - \frac{3\ell^3}{2}  = \frac{\epsilon^2}{4} -\frac{3\epsilon^3}{16}\]
for $\ell = \epsilon/2$.
It follows that 
\begin{equation}\label{eq:chernoffLB}\P\left(\frac{Z_1+\dots+Z_d}{d}\leq 1-\epsilon\right)\leq e^{-d\left(\frac{\epsilon^2}{4} -\frac{3\epsilon^3}{16}\right)} \leq 
e^{-d\left(\frac{\epsilon^2 - \epsilon^3}{4} \right)} 
\;.\end{equation}

\subsection{Application to the Johnson-Lindenstrauss lemma }\label{sec:JL}
Now that we have proved that squares of sub-Gaussian variables concentrate at least as well as squares of Gaussian variables, as far as the Chernoff bounds (\ref{eq:chernoffUB}) and (\ref{eq:chernoffLB}) are concerned, we recall for the sake of self-containment the argument that permits to obtain the Johnson-Lindenstrauss lemma with no assumption on the data: 
\begin{theorem}[Johnson-Lindenstrauss Lemma]\label{thm:classicalJL}
Let $n\geq 2$, $x_1,\dots,x_n\in\R^p$ and $\epsilon\in(0,1)$. For every $d\geq \frac{8\ln(n)}{\epsilon^2-\epsilon^3}$ there exists a matrix $A\in\mathcal{M}_{d,p}(\R)$ such that 
\begin{equation}\label{eq:JLprop_thm1} \forall 1\leq i,j\leq n, \quad (1-\epsilon) \|x_i-x_j\|^2\leq \|Ax_i-Ax_j\|^2 \leq (1+\epsilon) \|x_i-x_j\|^2 \;.\end{equation}
\end{theorem}
\begin{proof}
In the sequel, we assume that $A_{i,j} =  T_{i,j}/\sqrt{d}\;, 1\leq i\leq d,1\leq j\leq p$, where the $(T_{i,j})$ are centered, standard independent variables of a $1$-sub-Gaussian distribution $P$: 
\[\E[T_{i,j}] = 0, \quad \Var[T_{i,j}] = 1, \quad \E\left[e^{\lambda T_{i,j}}\right] \leq e^{\frac{\lambda^2}{2}} \;.\]
For a vector $y\in\R^p$ of norm $\|y\|_2 = 1$, define $Y = Ay$ and for all $i\in\{1,\dots,d\}$
\[Z_i = \sqrt{d} \;Y_i = \sum_{j=1}^p y_j T_{i,j}\;.\] 
Then, as  for all $\lambda\in\R$
\[\E\left[e^{\lambda Z_i}\right] = \prod_{j=1}^p \E\left[e^{\lambda y_j T_{i,j}}\right] \leq \prod_{j=1}^p e^{\frac{y_j^2\lambda^2}{2}}  = e^{\frac{\lambda^2}{2}}\;, \]
$Z_i$ is $1$-sub-Gaussian.
Since 
\[  \|Ay\|^2 = \frac{1}{d} \sum_{i=1}^d \left(\sqrt{d}Y_i\right)^2  = \frac{1}{d} \sum_{i=1}^d Z_i^2\;,\]
Equations~\ref{eq:chernoffUB} and~\ref{eq:chernoffLB} yield:
\begin{align*} 
\P\left( \|Ay\|^2 \notin \big[1-\epsilon, 1+\epsilon\big]  \right)
= \P\left(  \frac{1}{d} \sum_{i=1}^d Z_i^2 > 1+\epsilon \right) + \P\left(  \frac{1}{d} \sum_{i=1}^d Z_i^2 < 1-\epsilon \right)
\leq 2\,
e^{-d\left(\frac{\epsilon^2 - \epsilon^3}{4} \right)}  \leq \frac{2}{n^2}
\end{align*}
as soon as $d\geq \frac{8\ln(n)}{\epsilon^2-\epsilon^3}$. By the union bound and the above inequality applied to $y=\frac{x_i-x_j}{\|x_i-x_j\|}$ for all $i<j$ such that $x_i\neq x_j$,
\[\P\left( \bigcup_{1\leq i<j\leq n} \left\{ \big\|A(x_i-x_j)\big\|^2 \notin \left[ (1-\epsilon)\|x_i-x_j\|^2  ,  (1+\epsilon)\|x_i-x_j\|^2  \right] \right\}  \right) \leq \frac{n(n-1)}{n^2}<1\;,\]
hence giving the desired conclusion.
\end{proof}

Observe that the constant $8$ in Condition~(\ref{eq:JLcond}) is the best leading constant that can be obtained from this Chernoff and union-bound calculation (which already discards some information, such as the exact number of pairs). The dependency in $1/\epsilon^2$ also appears to be necessary, but the second-order term $\epsilon^3$ is slightly improvable. In the Gaussian case, the proof above allows one to use any
\[d\geq\frac{4\ln(n)}{\epsilon-\ln(1+\epsilon)}\;, \qquad\text{where}\qquad \frac{4\ln(n)}{\epsilon-\ln(1+\epsilon)} \leq \frac{8\ln(n)}{\epsilon^2} \left(1+\frac{\epsilon}{3}\right)^2\;,\]
as we saw in Section~\ref{subsec:UB}. For sub-Gaussian variables, the simple expression \ref{eq:JLcond} covers at the same time left- and right-deviations. 
Also observe that choosing $d\geq \frac{4\ln(n^2/\delta)}{\epsilon^2-\epsilon^3}$
permits Property \ref{eq:JLprop} to hold with probability at least $1-\delta$.

\subsection{What distribution should we use in random projections?}\label{sec:RP}

We have seen that any $1$-sub-Gaussian distribution of variance $1$ presents just the same guarantees as the standard Gaussian for random projections, in the sense that Theorem~\ref{thm:classicalJL} holds with the same sufficient dimension $8\ln(n)/(\epsilon^2-\epsilon^3)$; this does not mean that the tail probabilities themselves coincide (the comparison of the lower-tail exponential moments even fails, see Section~\ref{subsec:LB}). This is for example the case of $\displaystyle{P = \frac{\delta_{-1}+\delta_1}{2}}$, or of $ P = \mathcal{U}\Big(\big[-\sqrt{3},\sqrt{3}\big]\Big)$, which are very simple laws that are fast to sample from. Indeed, their exponential moment functions are $\cosh(\lambda)$ and $\sinh(\sqrt{3}\lambda)/(\sqrt{3}\lambda)$ respectively, which are upper-bounded by $e^{\lambda^2/2}$.
One may wonder, after Achlioptas in~\cite{Achlioptas01}, how \emph{sparse} a random projection matrix can be (sparse matrices require fewer computations). 
\begin{proposition}\label{prop:ArchiBest}
    If $X$ is a $1$-sub-Gaussian random variable of variance $1$, then $\P(X=0)\leq 2/3$, with equality if and only if $\P(X=-\sqrt{3}) = \P(X=\sqrt{3}) = \P(X\neq 0)/2$.
\end{proposition}
\begin{proof}
Let $q = \P(X\neq 0)$. If $X$ is $1$-sub-Gaussian with variance $1$, then $\E[X^4]\leq 3$ (see Section~\ref{subsec:LB}), and by the Cauchy-Schwarz inequality
\[1 = \E[X^2]^2 = \E\big[X^2\1_{\{X\neq 0\}}\big]^2 \leq \P(X\neq 0)\,\E[X^4] \leq 3q\;,\]
so that $q\geq 1/3$, that is $\P(X=0)\leq 2/3$. Moreover, $q = 1/3$ forces equality in the Cauchy-Schwarz inequality, that is, $X^2$ is almost surely constant on $\{X\neq 0\}$, equal to $1/q = 3$ since $\E[X^2]=1$; as $\E[X]=0$, this means that $\P(X=-\sqrt{3}) = \P(X=\sqrt{3}) = 1/6$. 
The choice 
\begin{equation}\label{eq:def_distribution}
	P = \frac{q}{2}\delta_{-\frac{1}{\sqrt{q}}} + (1-q)\delta_0 + \frac{q}{2}\delta_{\frac{1}{\sqrt{q}}}
\end{equation} 
with $q=1/3$ is indeed the suggestion of Achlioptas, and it is 
$1$-sub-Gaussian. The justification of this choice in~\cite{Achlioptas01} uses moment computations, while we here only need to check that for all $\lambda\in\R$, 
\[\E[e^{\lambda X}] = 1-q + q \cosh\left(\frac{\lambda}{\sqrt{q}}\right) = 1 + \sum_{k=1}^\infty \frac{\lambda^{2k}}{q^{k-1}(2k)!} \leq  e^{\lambda^2/2} = 1 + \sum_{k=1}^\infty \frac{\lambda^{2k}}{2^k\, k!}\]
whenever $q\geq 1/3$.
A sufficient condition for the inequality is that for all $k\geq 1$, 
\begin{equation}\label{eq:condition_subgaussian}
	\frac{1}{q^{k-1}(2k)!} \leq \frac{1}{2^k \,k!} \iff q^{k-1}\geq \frac{2^k\,k!}{(2k)!} \;.
\end{equation}
For $k=1$ this is always true, for $k=2$ it requires that $\displaystyle{q \geq \frac{4\times 2}{24} =\frac{1}{3}}$. A simple induction shows that if $q\geq 1/3$, the condition is also satisfied for all $k\geq 3$. 
Conversely, if $q<1/3$ then $\displaystyle{\E[e^{\lambda X}] -  e^{\lambda^2/2}  \sim_{\lambda\to 0}  \Big(\frac{1}{q}-3\Big)\frac{\lambda^4}{24} > 0}$, and $P$ is not $1$-sub-Gaussian. 
\end{proof}

This shows that Achlioptas' suggestion is the only "optimal" choice in terms of sparsity for a variance $1$ and $1$-sub-Gaussian distribution.
Nevertheless, many other choices are possible, such as for example $P = \frac{1}{12}\delta_{-2} +\frac{1}{6}\delta_{-1}+  \frac{1}{2}\delta_0 + \frac{1}{6}\delta_{1} + \frac{1}{12}\delta_{2}$. This law is symmetric with $\E[X^2]=1$, $\E[X^4]=3$ and $|X|\leq 2$, so that $\E[X^{2k+2}]\leq 4\,\E[X^{2k}]$ for all $k$, whereas the even moments $(2k-1)!!$ of the standard Gaussian satisfy $(2k+1)!! \geq 5\,(2k-1)!!$ for $k\geq 2$; by induction, $\E[X^{2k}]\leq (2k-1)!!$ for all $k$, and a term-by-term comparison of the series of $\E[e^{\lambda X}]$ and $e^{\lambda^2/2}$ shows that $P$ is $1$-sub-Gaussian.

\section{Very sparse random projections}\label{sec:sparseRP}

We say that a random matrix with independent entries is $q$-sparse if each coefficient has probability at least $1-q$ to be equal to zero; the expected proportion of non-zero entries of such a matrix is at most $q$ (this is a property of the distribution, not of each realization). In the previous section, we showed that the minimal probability $q$ for the nonzero values of a suitable $1$-sub-Gaussian distribution $P$ is $1/3$. In fact, this result was proven in \cite{verySparse06} with moment arguments. It allows one to take a target dimension $d \geq 8\ln(n)/(\epsilon^2 - \epsilon^3)$ -- see (\ref{eq:JLcond}) -- to get an $\epsilon$-quasi-isometry with non-zero probability, whatever the data $x$.

This does not exclude the possibility of using $q$-sparse projection matrices with $q<1/3$, however. 
Technically speaking, the previous analysis remains quite conservative in that the sub-Gaussianity of $Z_i$ is deduced from the sub-Gaussianity of \emph{each} of its summands. We may expect to gain a lot of sparsity by using the fact that a sum can be a lot more concentrated than each of its components. Figure~\ref{fig:expeFull}
suggests that, at least under certain conditions on the data, much sparser matrices may be considered.
\begin{figure}
    \includegraphics[width=0.8\textwidth]{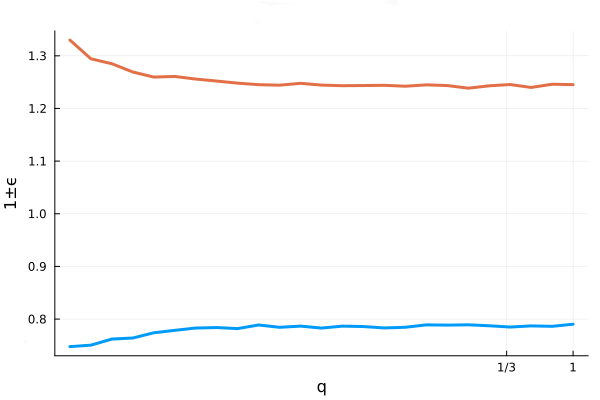}
    \caption{Extremal squared-distance ratios as a function of the sparsity parameter $q$ of the random projection entries. Each data point $x_i\in\R^{10000}$ has independent Gaussian entries. The projection matrix $A$ has independent entries with distribution $\frac{q}{2}\delta_{-\frac{1}{\sqrt{dq}}} + (1-q)\delta_0 + \frac{q}{2}\delta_{\frac{1}{\sqrt{dq}}}$. The target dimension is $d=500$. The blue line shows $r_{\min} = \min_{i<j} \frac{\|A(x_i-x_j)\|^2}{\|x_i-x_j\|^2}$, while the red line shows $r_{\max} = \max_{i<j} \frac{\|A(x_i-x_j)\|^2}{\|x_i-x_j\|^2}$, the extrema being taken over pairs with $x_i\neq x_j$; the smallest admissible tolerance in (\ref{eq:JLprop}) is $\max\{1-r_{\min}, r_{\max}-1\}$.\\
    In this simulation, the value $q=1/3$ plays no special role: much sparser matrices respect pairwise distances just as well.}
    \label{fig:expeFull}
\end{figure}

In this section, we present two results aimed at quantifying the minimum sparsity $q$ necessary to maintain the quasi-isometry condition with a dimension $d$ on the order of $\ln(n)/\epsilon^2$. In particular, Theorem \ref{th:jl_sparse} shows that $q$ can be as small as $\max_{i\neq j}\tfrac{\|x_i - x_j\|^2_{\infty}}{\|x_i - x_j\|_2^2}$. 
To finish, we establish that this order of magnitude cannot be improved in general: for Bernoulli-Rademacher projections and flat $s$-sparse vectors, Theorem~\ref{th:optimality} shows that $q \lesssim 1/s$ leads to failure with constant probability.

\subsection{Towards maximal sparsity}\label{subsec:min_sparsity}
Let  $U \in \R^{d \times p}$ be a matrix of iid $1$-sub-Gaussian entries with variance $1$, and $\zeta \in \R^{d \times p}$ be a matrix of iid Bernoulli variables of parameter $q\in(0,1]$ independent from $U$ that is used to \emph{mask} a proportion $1-q$ of the coefficients. We assume that for all $i,k$,
\begin{equation}\label{eq:sparse_matrix}
    A_{ik} = \frac{1}{\sqrt{dq}}\zeta_{ik} U_{ik} \;.
\end{equation}
For the coefficients of $U$, one can take Achlioptas' choice $\tfrac{1}{6}\delta_{-\sqrt{3}} +\tfrac{2}{3}\delta_0 + \tfrac{1}{6}\delta_{\sqrt{3}}$ to gain yet another fraction of sparsity on top of the mask: the entries of $A$ are then non-zero with probability $q/3$ only. In the sequel, $q$ always denotes the mask parameter. We apply the matrix $A$ to $n$ points $x_1, \dots, x_n$ in a high-dimensional space $\R^p$, and we look for the minimal conditions under which the quasi-isometry property (\ref{eq:JLprop}) still holds with positive probability.
We propose a first result in that direction.
\begin{theorem}\label{thm:JLs1s}
Let $n\geq 2$, $x_1,\dots,x_n\in\R^p$ and $\epsilon\in(0,1)$. For every $\displaystyle{d\geq \frac{36\ln\big(2n^2\big)}{\epsilon^2-\epsilon^3}}$ and every $q\in(0,1]$ such that 
\begin{equation}\label{eq:cond_q_1}
q\geq \max_{i\neq j}\frac{18 \|x_i-x_j\|_4^4  + 2\|x_i-x_j\|_\infty^2\|x_i-x_j\|_2^2}{\epsilon^2\|x_i-x_j\|_2^4} \ln(2dn^2)\;,\end{equation} it holds with positive probability that 
\begin{equation}\label{eq:JLprop_thm2} \forall 1\leq i,j\leq n, \quad (1-\epsilon) \|x_i-x_j\|^2\leq \|Ax_i-Ax_j\|^2 \leq (1+\epsilon) \|x_i-x_j\|^2 \; . \end{equation}
\end{theorem}

In particular, Theorem \ref{thm:JLs1s} establishes that the $q$-sparse random ensemble (\ref{eq:sparse_matrix}) satisfies the quasi-isometry condition with positive probability if $d \gtrsim \ln(n)/\epsilon^2$ and $q \gtrsim \max_{i\neq j}\tfrac{\|x_i - x_j\|^2_{\infty}}{\epsilon^2\|x_i - x_j\|_2^2}\ln(nd)$ up to a constant factor. Hence, if the coefficients of the differences $x_i - x_j$ for $i \neq j$ are of the same order of magnitude $1/\sqrt{p}$, then $q$ is allowed to be of order $\ln(nd)/(\epsilon^2p)$, which is much smaller than $1/3$ if $p \gg \epsilon^{-2}\ln(nd)$. Note that $q$ is a probability: Theorem~\ref{thm:JLs1s} provides a usable sparse matrix only when the right-hand side of (\ref{eq:cond_q_1}) is at most $1$. 
The cost in terms of target dimension is only a multiplicative constant (that is not optimized in the previous reasoning). The proof of Theorem \ref{thm:JLs1s} relies on similar ideas as in the preceding section, and will be provided in Section~\ref{sec:proofs}.

In Theorems~\ref{thm:classicalJL} and~\ref{thm:JLs1s}, the target dimension is of order $\ln(n)/\epsilon^2$. 
The interest of Theorem~\ref{thm:JLs1s} is mostly theoretical, as the constants can become large. It turns out that if we allow slightly larger target dimensions, of order $1/\epsilon^2$ up to logarithmic factors in $n$ and $1/\epsilon$, then we can decrease even further the sparsity parameter $q$. 
For the sake of completeness, we state this version of the quasi-isometry property (\ref{eq:JLprop}) with high probability instead of just with positive probability. 
\begin{theorem}\label{th:jl_sparse}
	Let $n\geq 2$, let $x_1, \dots, x_n$ be arbitrary vectors in $\R^p$, let $\epsilon \in (0,1)$, and let $A \in \R^{d \times p}$ be a random matrix with independent entries $A_{ik} = \frac{1}{\sqrt{dq}}\zeta_{ik} U_{ik}$ as in (\ref{eq:sparse_matrix}), with $q\in(0,1]$ such that
 \begin{equation}\label{eq:condition_q}
	q \geq \max_{i\neq j}\frac{\|x_i - x_j\|^2_{\infty}}{\|x_i - x_j\|_2^2} \; .
\end{equation}
Then for any $\delta \in (0,1)$ and any $d$ such that
\begin{equation}\label{eq:d_0}
	d \geq  \frac{12}{\epsilon^2} \ln(3n^2/\delta) \left( 1 + \sqrt{4\ln(3nd/\delta)} + 2\ln(3nd/\delta) \right)^2 \; ,
\end{equation}
the $\epsilon$-quasi-isometry property (\ref{eq:JLprop}) holds with probability at least $1-\delta$.
\end{theorem}
Let $d_0 = d_0(n,\delta,\epsilon)$ denote the least positive integer $d$ satisfying Condition~(\ref{eq:d_0}); note that the right-hand side of (\ref{eq:d_0}) depends on $d$ itself through $\ln(3nd/\delta)$, so that (\ref{eq:d_0}) is an implicit condition on $d$. Up to logarithmic factors in $n$, $1/\delta$ and $1/\epsilon$, $d_0$ is still of order $1/\epsilon^2$: for fixed $n$ and $\delta$, $d_0$ grows like $\epsilon^{-2}\ln^2(1/\epsilon)$ as $\epsilon\to 0$. 
The parameter $q$ can be chosen as small as $\max_{i\neq j}\frac{\|x_i - x_j\|^2_{\infty}}{\|x_i - x_j\|_2^2}$ while keeping the original guarantee of Johnson-Lindenstrauss (\ref{eq:JLprop}) with positive probability (and even with probability at least $1/2$) under the same condition (\ref{eq:JLcond}) up to logarithmic factors: we require $d \geq d_0(n, 1/2, \epsilon)$ instead of $d \geq 8\ln(n)/(\epsilon^2 - \epsilon^3)$. In comparison to Theorem \ref{thm:JLs1s}, we removed a factor of order $1/\epsilon^2$ in the minimal allowed sparsity $q$, at the cost of logarithmic factors in the target dimension.

\subsection{About the sparsity conditions (\ref{eq:condition_q}) and (\ref{eq:cond_q_1})}
For a non-zero vector $v\in\R^p$, define its \emph{effective sparsity}
\begin{equation}\label{eq:seff}
	s_{\mathrm{eff}}(v) := \frac{\|v\|_2^2}{\|v\|_\infty^2} \in [1, p]\;.
\end{equation}
Condition (\ref{eq:condition_q}) simply reads
\begin{equation}\label{eq:condition_suffisante_q}
	q \geq \max_{i\neq j} \frac{1}{s_{\mathrm{eff}}(x_i - x_j)}\;,
\end{equation}
and can be understood as a "not-too-high-sparsity" condition on the differences $x_i - x_j$. Indeed, if $v$ has at most $s$ non-zero coordinates then $s_{\mathrm{eff}}(v)\leq s$, with equality when these $s$ coordinates have the same magnitude; conversely, if at least $s$ coordinates of $v$ have magnitude at least $c\,\|v\|_\infty$ for some $c\in(0,1]$, then $s_{\mathrm{eff}}(v)\geq c^2 s$. Hence, if all the differences $x_i-x_j$, $i\neq j$, have at least $s$ coordinates of magnitude at least $c\,\|x_i-x_j\|_\infty$, then (\ref{eq:condition_q}) is implied by $q\geq 1/(c^2 s)$: an expected proportion $q \gtrsim 1/s$ of non-zero coefficients in $A$ suffices. This condition is for instance very weak in the dense case where all the coefficients of each difference $x_i - x_j$ are of the same order of magnitude, in the sense that $|x_{ik} - x_{jk}| \asymp \|x_i-x_j\|_\infty$ for all $k$: it then only requires an expected proportion of non-zero coefficients of order $q \gtrsim 1/p$. Similarly, since $\|v\|_4^4 \leq \|v\|_\infty^2\|v\|_2^2$, Condition (\ref{eq:cond_q_1}) is implied by $q \geq 20\ln(2dn^2)/\big(\epsilon^2 \min_{i\neq j} s_{\mathrm{eff}}(x_i-x_j)\big)$.

Condition (\ref{eq:condition_q}) becomes however much stronger when there exists a difference $x_i - x_j$ which is $s$-sparse for a small $s$, that is $|\{ k:~ x_{ik} \neq x_{jk}\}| \leq s$. Indeed, in such a sparse case, $s_{\mathrm{eff}}(x_i - x_j) \leq s$, and the condition $q \geq 1/s$ is necessary to satisfy (\ref{eq:condition_q}).

\subsection{Theoretical limit to the sparsity}\label{subsec:optimality}

\begin{figure}
    \includegraphics[width=0.24\textwidth]{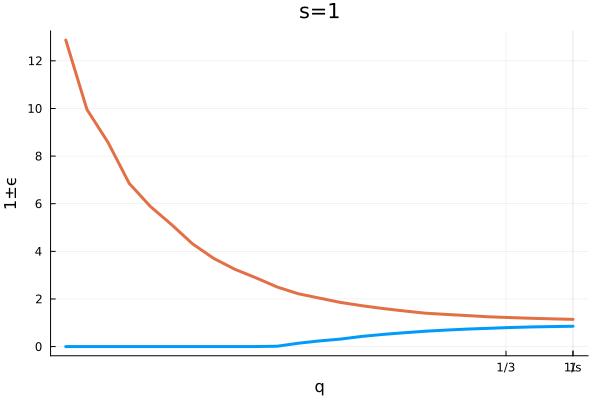}
    \includegraphics[width=0.25\textwidth]{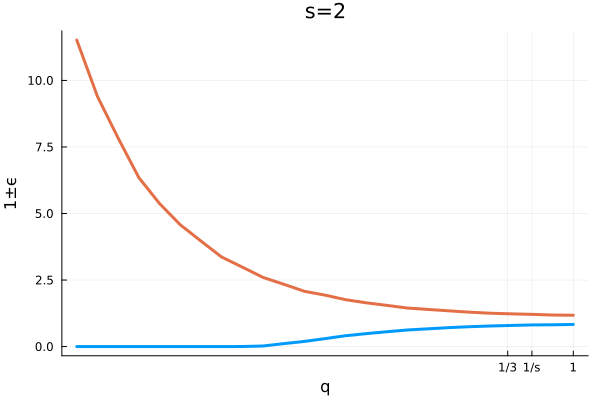}
    \includegraphics[width=0.24\textwidth]{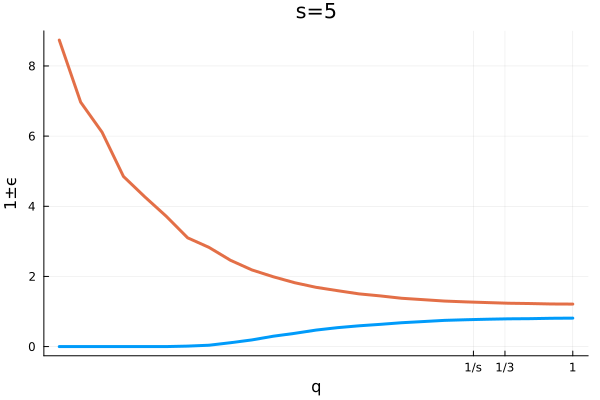}
    \includegraphics[width=0.24\textwidth]{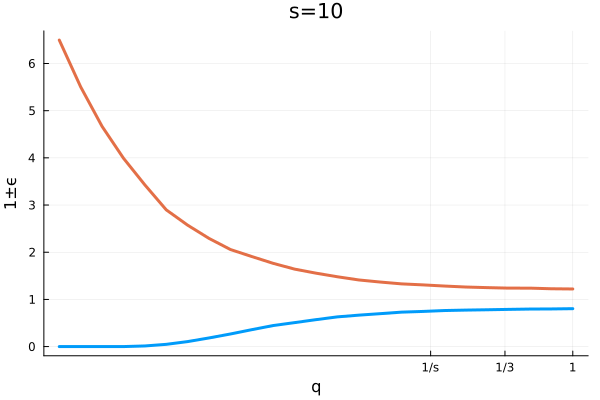}
    \includegraphics[width=0.24\textwidth]{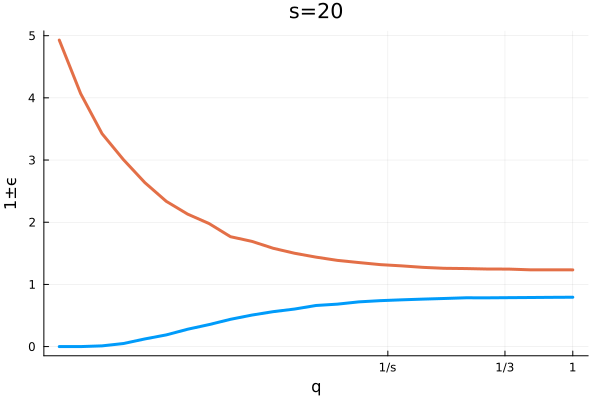}
    \includegraphics[width=0.25\textwidth]{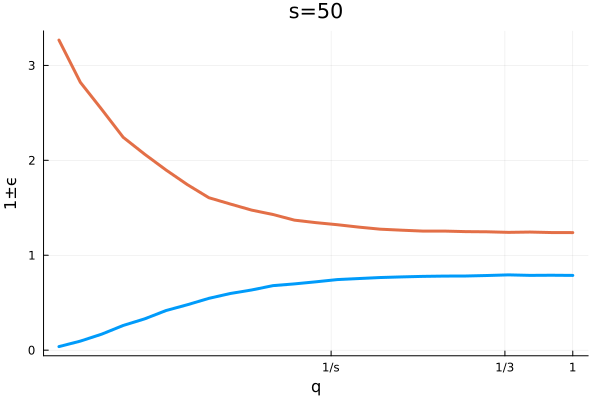}
    \includegraphics[width=0.24\textwidth]{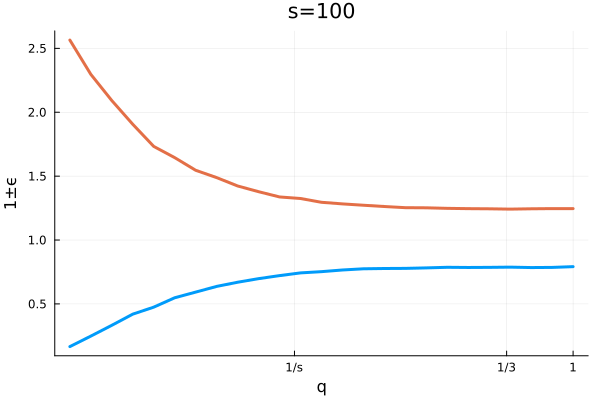}
    \includegraphics[width=0.24\textwidth]{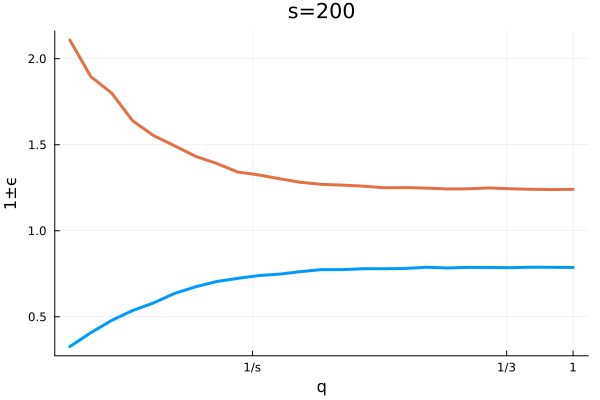}
    \includegraphics[width=0.24\textwidth]{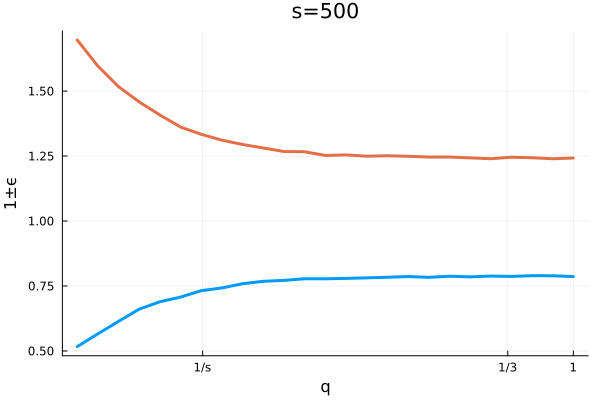}
    \includegraphics[width=0.24\textwidth]{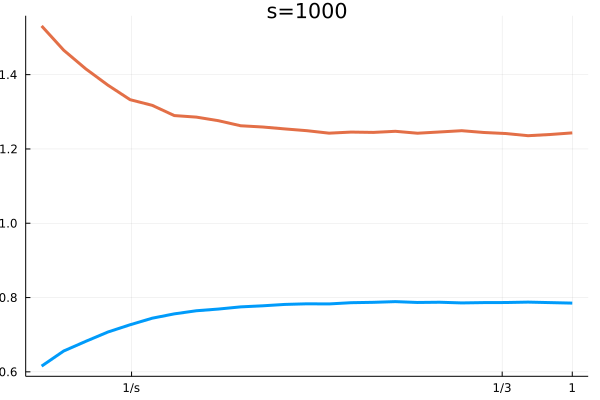}    
    \includegraphics[width=0.24\textwidth]{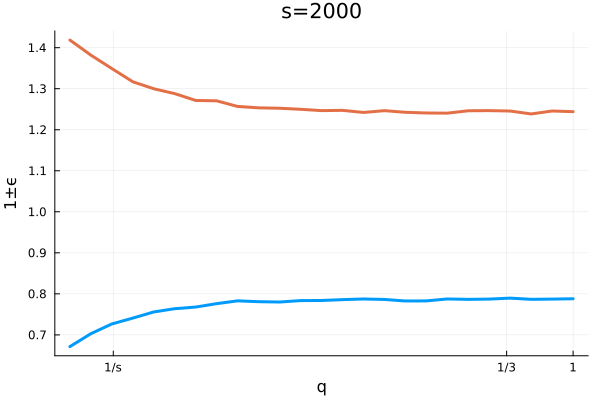}    
    \includegraphics[width=0.24\textwidth]{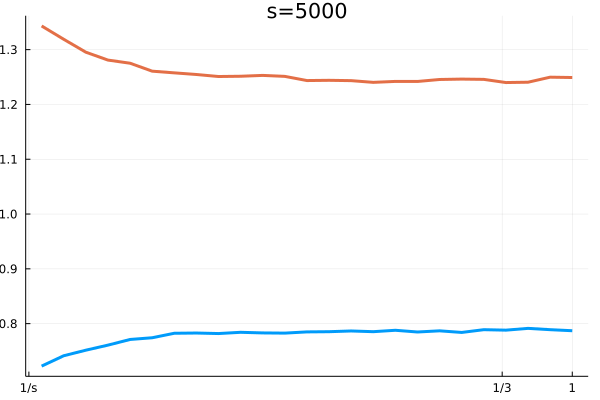}    
        \caption{Extremal squared-distance ratios $r_{\min}$ (blue) and $r_{\max}$ (red), defined as in Figure~\ref{fig:expeFull}, as a function of the sparsity parameter $q$ of the random projection entries (logarithmic scale), for different values of the sparsity $s$ of the data. Each data point $x_i\in\R^{10000}$ has exactly $s$ non-zero components, which are independent Gaussian entries. The coefficients of the projection matrix $A$ are independent and have distribution $\frac{q}{2}\delta_{-\frac{1}{\sqrt{dq}}} + (1-q)\delta_0 + \frac{q}{2}\delta_{\frac{1}{\sqrt{dq}}}$. The target dimension is $d=500$. Observe that the scales of the ordinates are different between the plots. Note that a difference $x_i-x_j$ may have up to $2s$ non-zero coefficients, and that it is the effective sparsity of the differences that matters in Theorems~\ref{thm:JLs1s} and~\ref{th:jl_sparse}.\\
    In these simulations, quasi-isometry holds whenever $q\times s$ is sufficiently large. }
    \label{fig:expeSparse}
\end{figure}

\begin{figure}
    \includegraphics[width=0.24\textwidth]{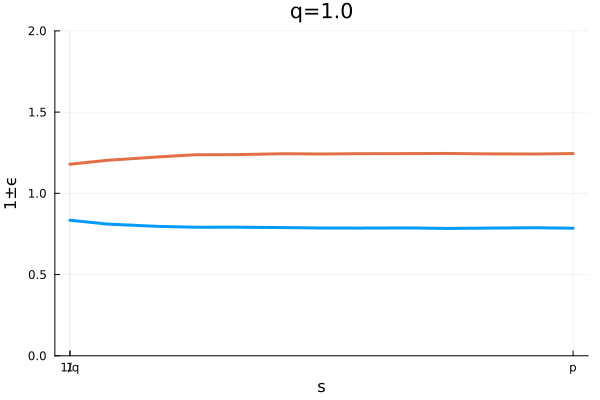}
    \includegraphics[width=0.24\textwidth]{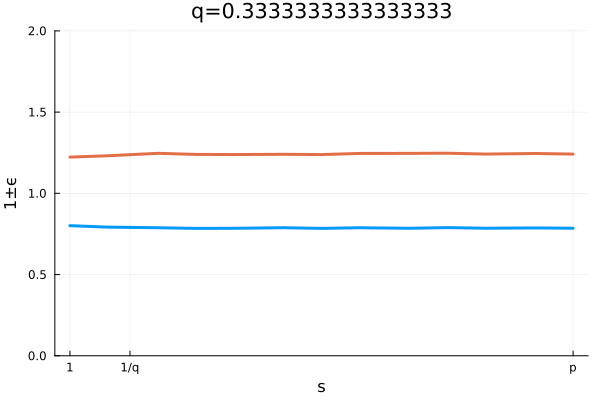}
    \includegraphics[width=0.24\textwidth]{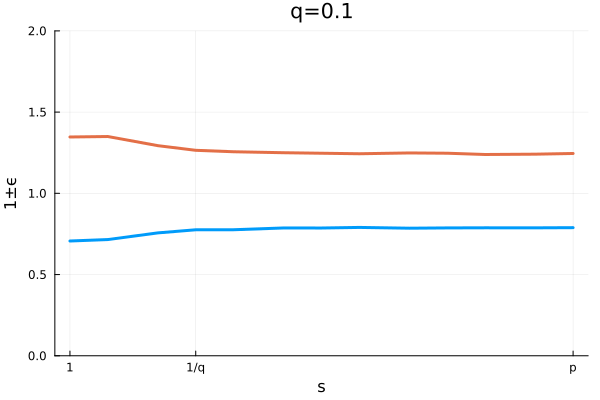}
    \includegraphics[width=0.25\textwidth]{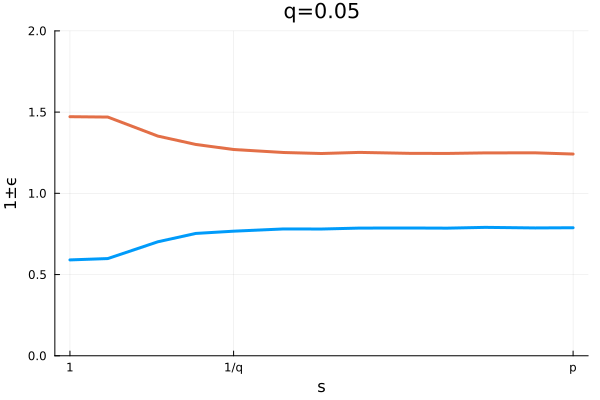}
    \includegraphics[width=0.24\textwidth]{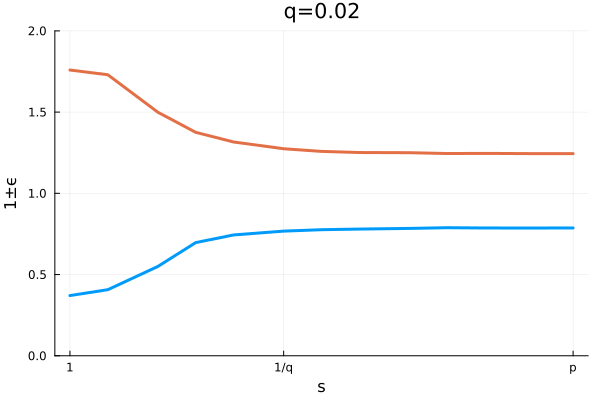}
    \includegraphics[width=0.24\textwidth]{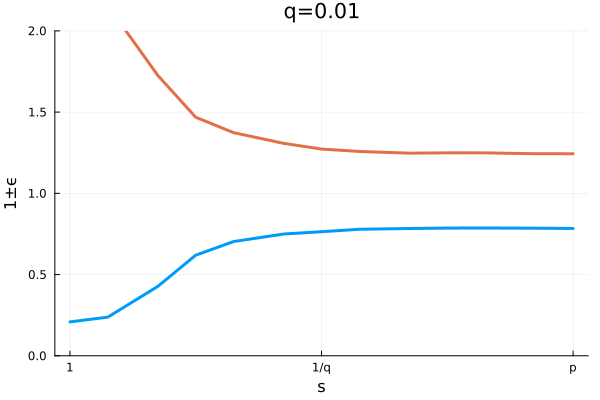}
    \includegraphics[width=0.24\textwidth]{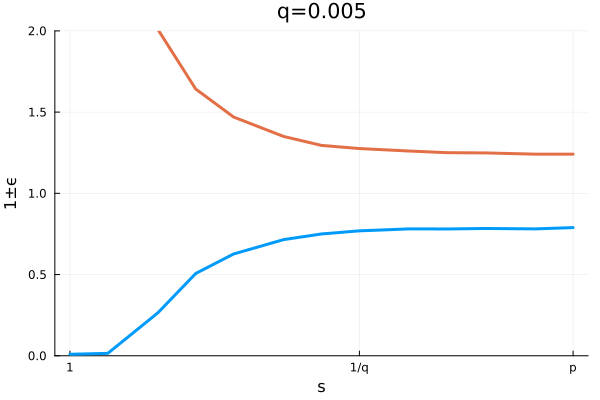}
    \includegraphics[width=0.25\textwidth]{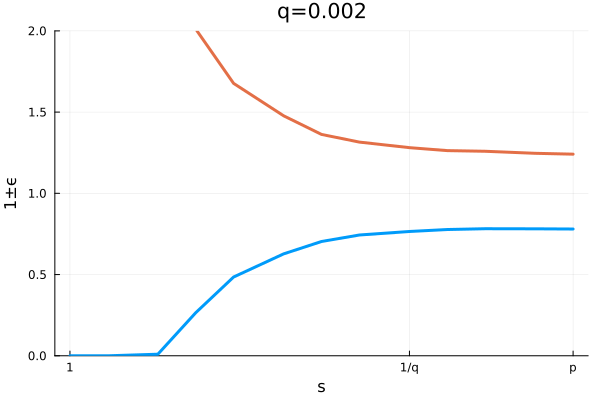}
    \includegraphics[width=0.24\textwidth]{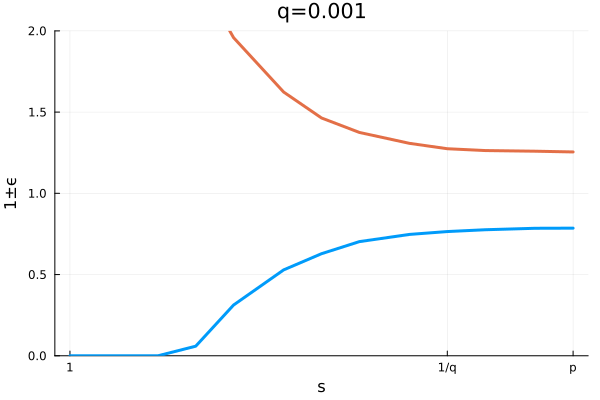}
    \includegraphics[width=0.24\textwidth]{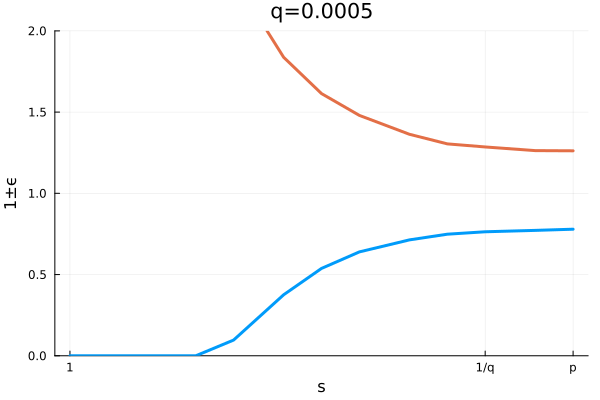}
    \includegraphics[width=0.24\textwidth]{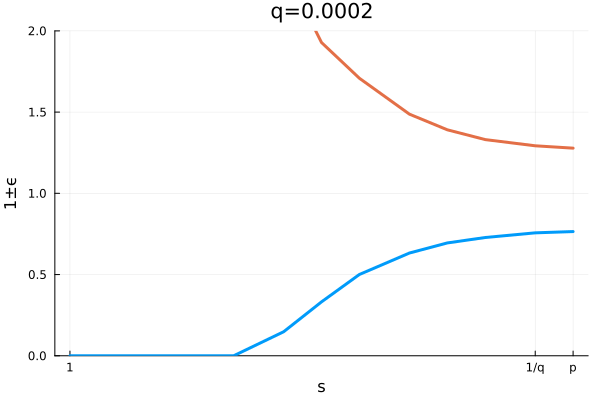}    
    \includegraphics[width=0.24\textwidth]{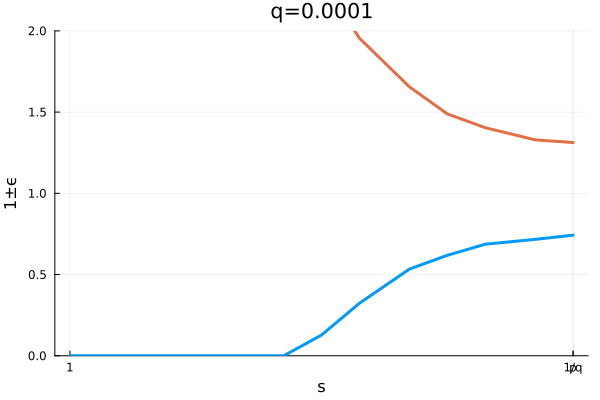}
    
    \caption{Extremal squared-distance ratios $r_{\min}$ (blue) and $r_{\max}$ (red), defined as in Figure~\ref{fig:expeFull} (over pairs with $x_i\neq x_j$), as a function of the expected sparsity $s$ of the data (logarithmic scale), for different values of the sparsity $q$ of the projection matrix. Each of the $n=100$ data points $x_i\in\R^{10000}$ has independent coefficients that are non-zero with probability $s/p$, so that $s$ is the expected number of non-zero coefficients of $x_i$; the non-zero coefficients are independent and uniformly distributed on $\{-1,+1\}$. The projection matrix $A$ has independent entries with distribution $\frac{q}{2}\delta_{-\frac{1}{\sqrt{dq}}} + (1-q)\delta_0 + \frac{q}{2}\delta_{\frac{1}{\sqrt{dq}}}$. The target dimension is $d=500$.\\
    In these simulations, the values $q\geq 1/3$ give a good quasi-isometry for all the data considered, as expected from Theorem~\ref{thm:classicalJL}. For smaller values of $q$, the number $s$ of non-zero coefficients needs to be larger than $1/q$. 
    }
    \label{fig:expeSparse_q}
\end{figure}

It turns out that the condition $q \gtrsim 1/s$ cannot be substantially weakened, at least for the Bernoulli-Rademacher ensemble and for flat $s$-sparse vectors, if we impose the dimension $d$ to be of order $1/\epsilon^2$ up to logarithmic factors. This observation was already highlighted in \cite{matouvsek08}, and we establish a result in this section that aligns with this perspective.
In the simulations of Figures~\ref{fig:expeSparse} and~\ref{fig:expeSparse_q}, random projections remain equally efficient as long as the expected proportion of non-zero coefficients is clearly above the minimum between $1/s$ and $1/3$ (note that these figures use the sparsity $s$ of the individual points, whereas the theorems involve the effective sparsity of the differences). 
Our lower bound is based on the following intuition: let $y$ be a unit vector with exactly $s$ non-zero coordinates, all equal to $\pm 1/\sqrt{s}$ (so that $s_{\mathrm{eff}}(y) = s$), and let $S = \{k \in [p]:~ y_k \neq 0\}$, where $[p] = \{1,\dots,p\}$. If $A \in \R^{d\times p}$ is any random matrix with independent coefficients such that $\P(A_{ik}\neq 0) \leq q$ for all $(i,k)$, then,
\[
	\P(Ay = 0) \geq \P\Big(\forall (i,k) \in [d]\times S, A_{ik}=0\Big) 
	\geq (1-q)^{ds} 
	= e^{ds\ln(1-q)}
	\geq e^{-ds\tfrac{q}{1-q}} \; .
\]
Hence, if $q \leq 1/(2ds)$ (and $q\leq 1/2$), then $\P(Ay = 0) \geq 1/e$. In other words, there is no hope to satisfy the quasi-isometry property (\ref{eq:JLprop}) with high probability if $q \leq 1/(2ds)$. This argument misses however the regime where $\epsilon^2/s \lesssim q \lesssim 1/s$ if $d \asymp 1/\epsilon^2$. The following theorem, which is specific to the Bernoulli-Rademacher ensemble and to flat $s$-sparse vectors, covers this regime: it shows that quasi-isometry fails with probability at least $1/20$ whenever $q \leq \min\{1/240,\, 1/(2d\epsilon^2)\}/s$, that is, whenever
\[ q \leq \frac{\min\{1/240,\, 1/(2L)\}}{s} \qquad \text{if } d = L/\epsilon^2\;.\]
It hence fills the gap between $\epsilon^2/s$ and $1/s$ when $d \asymp 1/\epsilon^2$ up to logarithmic factors.

\begin{theorem}\label{th:optimality}
	Assume that $A \in \R^{d \times p}$ has iid coefficients $A_{ik} = \frac{1}{\sqrt{dq}}\zeta_{ik}U_{ik}$, where $\zeta_{ik}\sim\mathrm{Bern}(q)$ and $U_{ik}\sim\mathcal{U}(\{-1,1\})$ are independent, that is, $\sqrt{d}A_{ik}$ follows Achlioptas' distribution with parameter $q$ -- see Eq.~(\ref{eq:def_distribution}). Let $s\in\{1,\dots,p\}$, let $y \in \R^p$ be a unit vector with coordinates in $\{-1/\sqrt{s}, 1/\sqrt{s}, 0\}$, and let $\epsilon\in(0,1)$. If $dqs \epsilon^2 \leq 1/2$ and $qs \leq 1/240$, then 
		\begin{align*}
			\P(\|Ay\|^2 \in [1-\epsilon, 1+\epsilon]) \leq \frac{19}{20}\; .
		\end{align*}
\end{theorem}

In other words, if $d \asymp 1/\epsilon^2$ up to logarithmic factors, Theorem \ref{th:optimality} shows that the random projection fails with probability at least $1/20$ as soon as $q \lesssim 1/s$ up to logarithmic factors. It does not rule out success with positive probability for larger constants, nor the existence of good deterministic sparse matrices, and it concerns a specific ensemble and specific vectors. The constant $19/20$ is not optimized.\footnote{The published version of this theorem stated the weaker bound $1-e^{-5000}$; its proof contained a gap, which is fixed here by a shorter argument.} The proof, which is given at the end of Section~\ref{sec:proofs}, treats separately the two regimes $dqs\leq 1$ and $dqs\geq 1$. In the first regime, $Ay=0$ with probability at least $e^{-240/239}>1/3$. In the second regime, the fluctuations of $\|Ay\|^2$ around $1$ are of order $1/\sqrt{dqs} \geq \sqrt{2}\,\epsilon$, and the conclusion follows from the Paley-Zygmund inequality applied to $(\|Ay\|^2-1)^2$, together with an explicit control of the moments of order $2$, $4$ and $8$ of the coordinates of $Ay$.

\section{En passant: concentration of positive semidefinite quadratic forms}\label{sec:HW}
As a conclusion to this contribution, we highlight that the upper bound given in Section~\ref{subsec:min_sparsity} is in fact strongly connected to the Hanson-Wright inequality for sub-Gaussian random variables -- see e.g. \cite{rudelson2013hanson}, and \cite{simpleJL24} for an application to the Johnson-Lindenstrauss lemma. This inequality is known with precise constants for quadratic forms of Gaussian vectors -- see Example 2.12 in \cite{boucheron2013concentration} -- and it has been generalized with non-explicit constants to sub-Gaussian vectors, e.g. in Theorem 6.2.1 of ~\cite{vershynin19HDP}. In the case where the quadratic form is positive semidefinite, \cite{hsu2012tail} established an upper-tail bound with the same constants as in the Gaussian case. 
For completeness, we propose a succinct proof of this upper-tail Hanson-Wright-type bound for positive semidefinite quadratic forms of sub-Gaussian vectors. 

A random vector $X \in \R^d$ is said to be $1$-sub-Gaussian if, for any $u\in\R^d$, $\E[e^{u^T X}] \leq \exp\big(\|u\|^2_2/2\big)$.
In particular, if $Z_{1}, \dots, Z_d$ are independent $1$-sub-Gaussian real random variables, that is $\E[e^{\lambda Z_i}] \leq e^{\lambda^2/2}$, then for any orthogonal matrix $P$, $PZ$ is a $1$-sub-Gaussian vector. In contrast to \cite{rudelson2013hanson}, we do not require in the following proposition the coordinates of $X$ to be independent. 

\begin{proposition}[See also Theorem 2.1 of \cite{hsu2012tail}]\label{prop:right_tail}
	Let $S$ be any $d \times d$ symmetric positive semidefinite matrix, let $\|S\|_{op}$ be its largest eigenvalue and $\|S\|_F$ its Frobenius norm (with the convention $1/\|S\|_{op} = +\infty$ if $S=0$), and let $X\in\R^d$ be a $1$-sub-Gaussian vector.
    Then, for any $\ell \in [0,1/(2\|S\|_{op}))$,
    \begin{align*}
        \E_X[e^{\ell X^T S X}] \leq \exp\left(\ell  \Tr(S) + \tfrac{\ell ^2 \|S\|_F^2}{1-2\ell \|S\|_{op}}\right) \; .
    \end{align*}
\end{proposition}
As a consequence of Proposition \ref{prop:right_tail} and of Markov's inequality, it holds that with probability at least $1- \delta$,
    \begin{align*}
        X^T S X \leq \Tr(S) + \sqrt{4\|S\|_F^2\ln(1/\delta)} + 2\|S\|_{op}\ln(1/\delta) \; ,
    \end{align*}
for any $\delta \in (0, 1)$. Indeed, if $S\neq 0$, writing $t=\ln(1/\delta)$, $a = \|S\|_F$ and $b=\|S\|_{op}$, the choice $\ell = \sqrt{t}/(a+2b\sqrt{t}) < 1/(2b)$ gives $1-2\ell b = a/(a+2b\sqrt{t})$, hence $\ell^2a^2/(1-2\ell b) = ta/(a+2b\sqrt{t})$ and
\begin{align*}
\P\big(X^TSX \geq \Tr(S) + 2a\sqrt{t} + 2bt\big) &\leq \exp\left(-\ell\big(2a\sqrt{t}+2bt\big) + \frac{\ell^2a^2}{1-2\ell b}\right) \\
&= \exp\left(-\frac{2t(a+b\sqrt{t}) - ta}{a+2b\sqrt{t}}\right) = e^{-t} = \delta\;.
\end{align*}
In comparison to Theorem 6.2.1 of \cite{vershynin19HDP}, the constants are explicit -- they are indeed the same as in the Gaussian case. The argument applies however only to positive semidefinite matrices, and it only provides an upper tail relative to $\Tr(S)$: without an isotropy assumption, $\E[X^TSX]$ need not be equal to $\Tr(S)$, and no lower tail is given.
Proposition \ref{prop:right_tail} can be deduced from the proof of Theorem 2.1 in \cite{hsu2012tail} in the case $\mu=0$ and $\sigma = 1$, but we still provide our own short proof as the underlying ideas are at the core of the upper bounds of this paper.

\begin{proof}[Proof of Proposition \ref{prop:right_tail}]
	Let us write $S = P\,\Diag(\mu_1, \dots, \mu_d)\,P^T$, where $\mu_1 \geq \dots \geq \mu_d \geq 0$ and $P$ is an orthogonal matrix. Let also $Y$ be the sub-Gaussian vector equal to $P^TX$, and $\ell > 0$. By Fubini's theorem, 
	\begin{align*}
		\E_X\left[e^{\ell X^T S X}\right] &= \E_X\left[e^{\sum_{i=1}^d \ell \mu_i Y_i^2}\right] = \E_G \left[\E_X\left[e^{\sum_{i=1}^d \sqrt{2\ell \mu_i} Y_iG_i}\right]\right] \leq \E_G \left[e^{\sum_{i=1}^d \ell \mu_iG_i^2}\right] \; ,
	\end{align*}
where $G_1, \dots, G_d$ are independent standard and centered Gaussian random variables. Then,
\begin{align*}
    \E_G \left[e^{\sum_{i=1}^d \ell \mu_iG_i^2}\right] 
    &= \prod_{i=1}^d \frac{1}{\sqrt{1 - 2\ell  \mu_i}} 
    = e^{\ell  \Tr(S)}\prod_{i=1}^d\exp\left(- \tfrac{1}{2}\ln (1 - 2\ell \mu_i) - \ell  \mu_i\right)\\
    &\leq \exp\left(\ell  \Tr(S) + \sum_{i=1}^d \tfrac{\ell ^2 \mu_i^2}{1-2\ell\mu_i}\right) \;,
\end{align*}
where the last inequality comes from the inequality $-\tfrac{1}{2}\ln(1-2\ell \mu_i) - \ell \mu_i = \int_0^{\ell}\tfrac{2s\mu_i^2}{1-2s\mu_i}ds \leq \tfrac{\ell^2\mu^2_i}{1-2\ell\mu_i}$.
We conclude the proof by remarking that $\sum_{i=1}^d \tfrac{\ell ^2 \mu_i^2}{1-2\ell\mu_i} \leq \tfrac{\ell ^2 \|S\|_F^2}{1-2\ell\|S\|_{op}}$.
 
\end{proof}

\subsubsection*{Acknowledgments}
{\sloppy Aur\'elien Garivier thanks the Chaire SeqALO (ANR-20-CHIA-0020-01) and the PEPR IA project FOUNDRY (ANR-23-PEIA-0003) for their support. The authors are thankful to Pierre Bellec and Nicolas Verzelen for their precious feedback.\par}

\appendix
\section{Proofs of Theorems~\ref{thm:JLs1s}, ~\ref{th:jl_sparse} and  ~\ref{th:optimality}}\label{sec:proofs}
\noindent
While the proof of Theorem~\ref{thm:JLs1s} remains as close as possible to that of Theorem~\ref{thm:classicalJL}, it provides some intuitions for the proof of Theorem~\ref{th:jl_sparse}. 
\subsection{Proof of Theorem \ref{thm:JLs1s}}
Let  $U \in \R^{d \times p}$ be a matrix of iid $1$-sub-Gaussian entries with variance $1$, and $\zeta \in \R^{d \times p}$ be a matrix of iid Bernoulli variables of parameter $q$ independent from $U$ that is used to \emph{mask} a proportion $1-q$ of the coefficients. We assume that for all $i,k$,
\begin{equation}\label{eq:sparse_matrix_app}
    A_{ik} = \frac{1}{\sqrt{dq}}\zeta_{ik} U_{ik} \;,
\end{equation}
and for any vector $y\in\R^p$ of norm $\|y\|_2=1$ write as before $Y = Ay$ and $Z_i=\sqrt{d}Y_i$.

Conditionally on the mask $\zeta$, $Z_i = \sum_{k=1}^p y_k \zeta_{ik}U_{ik}/\sqrt{q}$ is a sum of independent centered variables, with conditional variance
\[v_i := \Var[Z_i\,|\,\zeta] = \sum_{k=1}^p \frac{y_k^2}{q}\zeta_{ik}\;,\]
and, as in the proof of Theorem~\ref{thm:classicalJL}, for all $\lambda\in\R$
\[\E\left[\exp\left(\lambda Z_i\right)\,\middle|\,\zeta \right] = \prod_{k=1}^p \E\left[\exp\left(\lambda \frac{y_k\zeta_{ik}U_{ik}}{\sqrt{q}}\right)\,\middle|\, \zeta\right] \leq \prod_{k=1}^p  \exp\left( \frac{\lambda^2y_k^2\zeta_{ik}}{2q}\right) = \exp\left(\frac{\lambda^2 v_i}{2}\right)\;.\]
Since $\E[v_i] = \|y\|^2 = 1$, the idea is that $v_i$ is typically close to $1$. For $a := \epsilon/3$ and for each $i\in\{1,\dots,d\}$, let
\[E_i = \left\{ 1-a\leq v_i \leq 1+a\right\} \;.\]
By Bernstein's inequality applied to the $\left[0, \frac{\|y\|_\infty^2}{q}\right]$-valued independent variables $\left(\frac{y_k^2}{q}\zeta_{ik}\right)_{1\leq k\leq p}$, which have variance $\frac{y_k^4}{q^2}\;q(1-q)\leq \frac{y_k^4}{q}$, 
\begin{align*}
\P(\bar{E_i}) \leq 2 \exp\left(-\frac{\epsilon^2 /18}{\sum_{k=1}^p \frac{y_k^4}{q}  + \frac{\|y\|_\infty^2 \epsilon}{9q}}\right) \leq 2\exp\left(-\frac{q\epsilon^2}{18 \|y\|_4^4  + 2\|y\|_\infty^2}\right)\;,
\end{align*}
which is at most $1/(dn^2)$ as soon as 
\begin{equation}\label{eq:qmin}
q\geq \frac{18 \|y\|_4^4  + 2\|y\|_\infty^2}{\epsilon^2} \ln(2dn^2)\;.
\end{equation}
Let $E=\bigcap_{i=1}^d E_i$. Conditionally on a mask $\zeta\in E$, all the $v_i$ are positive and the normalized variables $W_i := Z_i/\sqrt{v_i}$, $1\leq i\leq d$, are independent, centered, with variance $1$ and $1$-sub-Gaussian. They need not be identically distributed, but the proof of Proposition~\ref{prop:devX2} only uses the same bounds on the exponential moments of each term, and hence applies unchanged: for every $\zeta\in E$,
\[\P\left(\frac{1}{d}\sum_{i=1}^d W_i^2 \notin [1-a, 1+a] \,\middle|\, \zeta\right) \leq 2 e^{-d\left(\frac{a^2-a^3}{4}\right)} \leq 2e^{-d\left(\frac{\epsilon^2-\epsilon^3}{36}\right)}\;.\]
Moreover, on $E$, if $\frac{1}{d}\sum_{i=1}^d W_i^2 \in [1-a,1+a]$ then
\[(1-a)^2 \leq \frac{1}{d}\sum_{i=1}^d v_i W_i^2 = \frac{1}{d}\sum_{i=1}^d Z_i^2 = \|Ay\|^2 \leq (1+a)^2\;,\]
and $[(1-a)^2, (1+a)^2] \subset [1-\epsilon, 1+\epsilon]$ for $a=\epsilon/3$ and $\epsilon\in(0,1)$. Integrating over $\zeta\in E$, we obtain that
\begin{align*} 
\P\big( E\cap \left\{\|Ay\|^2 \notin \big[1-\epsilon, 1+\epsilon\big] \right\} \big)
\leq 2\,e^{-d\left(\frac{\epsilon^2 - \epsilon^3}{36} \right)} \;.
\end{align*}
Consequently, denoting by $E^{i,j}$ the set $E$ corresponding to $y=\frac{x_i-x_j}{\|x_i-x_j\|}$,
\begin{align*}
   & \P\left( \bigcup_{1\leq i<j\leq n} \left\{ \big\|A(x_i-x_j)\big\|^2 \notin \left[ (1-\epsilon)\|x_i-x_j\|^2  ,  (1+\epsilon)\|x_i-x_j\|^2  \right] \right\}  \right) 
\\ & \leq  \sum_{1\leq i<j\leq n } \left[ \P\Big(\overline{E^{i,j}}\Big) +  \P\left( E^{i,j}\cap \left\{\frac{\|A(x_i-x_j)\|^2}{\|x_i-x_j\|^2} \notin \big[1-\epsilon, 1+\epsilon\big] \right\} \right)\right]
<  \frac{1}{2}+ n^2 \;e^{-d\left(\frac{\epsilon^2 - \epsilon^3}{36} \right)}  \leq 1
\end{align*}
as soon as $q$ satisfies Eq.~(\ref{eq:qmin}) and $d\geq \frac{36\ln\big(2n^2\big)}{\epsilon^2-\epsilon^3}$.

\subsection{Proof of Theorem \ref{th:jl_sparse}}
	
Let $\odot$ be the Hadamard product, so that $A = \tfrac{1}{\sqrt{dq}}\zeta \odot U$. We assume that $y$ is a unit vector of $\R^p$ representing one of the unit vectors $\tfrac{x_i - x_j}{\|x_i - x_j\|}$, and we write as before $Y = Ay$. The coefficients of $w_i = \tfrac{1}{\sqrt{dq}}\zeta_{i\cdot} \odot y$ are equal to $w_{ik} = \tfrac{1}{\sqrt{dq}}\zeta_{ik}y_k$, and
\begin{equation*}
    Y_i^2 = \tfrac{1}{dq}\sum_{k',k}\zeta_{ik}\zeta_{ik'}U_{ik}U_{ik'}y_ky_{k'} =  (U_{i\cdot}^T w_i)^2\; .
\end{equation*}
Note that $d\|w_i\|^2 = \frac{1}{q}\sum_{k=1}^p y_k^2\zeta_{ik}$ is the conditional variance $v_i$ of $\sqrt{d}\,Y_i$ introduced in the proof of Theorem~\ref{thm:JLs1s}. While that proof required every $v_i$ to be close to $1$, we only control here the maximum and the sum of the $\|w_i\|^2$, which is what allows a smaller $q$ at the price of a larger $d$.

\medskip

\noindent
{\bf The upper bound}

\medskip

In the sequel, $\E_U$ and $\P_U$ denote expectation and probability conditionally on the mask $\zeta$, that is, with respect to $U$ only. Conditionally on $\zeta$, $Y_i = U_{i\cdot}^T w_i$ is centered and sub-Gaussian with variance proxy $\|w_i\|^2$, that is, $\E_U[e^{\lambda Y_i}] \leq e^{\lambda^2\|w_i\|^2/2}$ for all $\lambda\in\R$ (and $Y_i=0$ if $w_i=0$). Hence, if $G$ is a standard Gaussian random variable independent of everything else, it holds conditionally on $\zeta$ that for any $\ell$ in $[0, 1/(2\max_i\|w_i\|^2))$,
\[
	\E_U\left[e^{\ell Y_i^2}\right] = \E_G \left[\E_U\left[e^{\sqrt{2\ell} Y_iG}\right]\right]
	\leq \E_G \left[e^{\ell \|w_i\|^2G^2}\right]
	= \frac{1}{\sqrt{1-2\ell \|w_i\|^2}}
	\leq \exp\left(\ell  \|w_i\|^2 + \tfrac{\ell ^2 \|w_i\|^4}{1-2\ell\|w_i\|^2} \right) \; ,
\]
where the last inequality comes from the fact that $-\tfrac{1}{2}\ln(1-2\ell \|w_i\|^2) - \ell \|w_i\|^2 = \int_0^{\ell}\tfrac{2s \|w_i\|^4}{1-2s \|w_i\|^2}ds \leq \tfrac{\ell^2\|w_i\|^4}{1-2\ell \|w_i\|^2}$. Hence, by Markov's inequality and by the conditional independence of $Y_1,\dots,Y_d$ given $\zeta$, for any $\ell$ in $[0, 1/(2\max_i\|w_i\|^2))$,
\begin{equation}\label{eq:right_tail_jl}
	\P_U(\|Ay\|^2 \geq 1+ \epsilon) \leq \exp\left(\ell \sum_{i=1}^d \|w_i\|^2 + \frac{d\ell^2\max_i\|w_i\|^4}{1-2\ell \max_i\|w_i\|^2}-\ell(1+\epsilon)\right) \; .
\end{equation}
Let us now integrate according to $\zeta$. The $w_{ik}$'s are independent random variables, $w_{ik}$ being distributed as $\tfrac{y_k}{\sqrt{dq}}\zeta_{ik}$ with $\zeta_{ik}\sim\mathrm{Bern}(q)$. Moreover, $\Var(w_{ik}^2) \leq \tfrac{y_k^4}{d^2q^2}\E[\zeta_{ik}^4] \leq \tfrac{1}{d^2q}y_k^4$. Bernstein's inequality together with a union bound over the $d$ possible indices $i = 1, \dots, d$ gives that with probability at least $1- \delta/(3n^2)$,
\begin{equation*}
	\max_i\|w_i\|^2 \leq \tfrac{1}{d} + \sqrt{2\tfrac{1}{d^2q}\|y\|_4^4\ln(3n^2d/\delta)} + \tfrac{1}{dq}\|y\|_{\infty}^2\ln(3n^2d/\delta) \; .
\end{equation*}
The assumption (\ref{eq:condition_q}) implies that $q \geq \|y\|_\infty^2$. Since $\|y\|^2=1$, $\|y\|_4^4 \leq \|y\|_\infty^2$ and the following event $E$ holds with probability at least $1-\delta/(3n^2)$:
\begin{equation}\label{eq:event_xi}
	E = \left\{ \max_i\|w_i\|^2 \leq\tfrac{1}{d}\left(1 + \sqrt{4\ln(3nd/\delta)} + 2\ln(3nd/\delta)\right) \right\} = \left\{ \forall i, \|w_i\|^2 \leq \frac{\Psi}{d}\right\}\;.
\end{equation}
where for simplicity we write $\Psi=1 + \sqrt{4\ln(3nd/\delta)} + 2\ln(3nd/\delta)$.
Using the inequality $e^u \leq 1 + u + (e-2)u^2$ for any $u \in [0,1]$, we have that for any $\ell \in [0, \tfrac{dq}{\|y\|^2_{\infty}})$,

\begin{align*}
	\E\left[\exp\left(\ell\sum_{i=1}^d\|w_i\|^2\right)\right] &= \prod_{i,k}\left( q \exp(\tfrac{\ell}{dq}y_k^2) + 1-q \right)\\
	&\leq \prod_{i,k} \exp\left(q (\exp(\tfrac{\ell}{dq}y_k^2)-1)\right)\\
	&\leq \exp(\ell + (e-2)\tfrac{\ell^2}{dq}\|y\|_4^4)\\
	&\leq \exp(\ell + (e-2)\tfrac{\ell^2}{d}) \; .
\end{align*}

Let us now integrate the conditional probability $\P_U(\|Ay\|^2 \geq 1+ \epsilon)$ over $\zeta$. For any $\ell \in [0, d/(4\Psi))$, we have

\begin{align*}
	\P(\|Ay\|^2 \geq 1+ \epsilon) &\leq \E\left[\exp\left(\ell \sum_{i=1}^d \|w_i\|^2 + \frac{d\ell^2\max_i\|w_i\|^4}{1-2\ell \max_i\|w_i\|^2}-\ell(1+\epsilon)\right)\mathbf 1_{E}\right] + \frac{\delta}{3n^2}\\
	&\leq \E\left[\exp\left(\ell \sum_{i=1}^d \|w_i\|^2 + 2\ell^2\tfrac{\Psi^2}{d}-\ell(1+\epsilon)\right)\right] + \frac{\delta}{3n^2}\\
	&\leq \exp\left((e-2)\tfrac{\ell^2}{d} + 2\ell^2\tfrac{\Psi^2}{d} - \ell \epsilon\right)+\frac{\delta}{3n^2} \\
\end{align*}

The first inequality comes from Equation~(\ref{eq:right_tail_jl}) applied on the event $E$ defined in (\ref{eq:event_xi}). The second one comes from the fact that on $E$, $\max_i\|w_i\|^2\leq \Psi/d$ and $1-2\ell\max_i\|w_i\|^2 \geq 1/2$ since $\ell < d/(4\Psi)$. The third inequality comes from the fact that $\ell < d/(4\Psi)\leq d \leq dq/\|y\|^2_{\infty}$ and the above upper bound on $\E[\exp(\ell\sum_{i=1}^d\|w_i\|^2)]$.

Choosing $\ell = d\epsilon/(2(e-2) + 4\Psi^2) \leq d/(4\Psi^2)$, we get
\begin{align*}
	\P(\|Ay\|^2 \geq 1+ \epsilon) &\leq \exp\left(-\frac{d\epsilon^2}{4(e-2) + 8\Psi^2}\right)+\frac{\delta}{3n^2} \; .
\end{align*}

Hence, since $4(e-2)+8\Psi^2 \leq 12\Psi^2$, if $d \geq d_0 = 12\ln(3n^2/\delta)\Psi^2/ \epsilon^2$ we obtain that 
\begin{equation*}
	\P(\|Ay\|^2 \geq 1+ \epsilon) \leq \exp\left(-\ln\frac{3n^2}{\delta}\right) + \frac{\delta}{3n^2} =  \frac{2\delta}{3n^2}\; .
\end{equation*}
A union bound over all the $n(n-1)/2 \leq n^2$ pairs gives that 
\begin{equation}\label{eq:ub_sparse}
	\P\left( \bigcup_{1\leq i<j\leq n} \left\{ \big\|A(x_i-x_j)\big\|^2 \geq (1+ \epsilon)\|x_i-x_j\|^2 \right\}  \right) \leq 2\delta/3\;.
\end{equation}

\medskip
\noindent
{\bf The lower bound}

\medskip

For the lower bound, we use the same arguments as in Section~\ref{subsec:LB}. We still have that $\E [(Ay)_i^2] = 1/d$ for any $i \in \{1, \dots, d\}$, but since the variables $(Ay)_i$ need not be sub-Gaussian with variance proxy $1/d$ (they are sub-Gaussian with variance proxy $1/(dq)$ only), we do not have the bound $\E[(Ay)_i^4] \leq 3/d^2$. Instead, we bound the fourth moment as follows:

\begin{align*}
	\E[(Ay)_i^4] &\leq \frac{3}{d^2q^2} \sum_{k\neq k'} q^2 y_k^2y_{k'}^2 + \frac{1}{d^2q^2} \sum_{k=1}^p qy_k^4 \E[U_{ik}^4]\\
	&\leq \frac{3}{d^2} + \frac{3\|y\|_{\infty}^2}{d^2q} \leq \frac{6}{d^2} \; .
\end{align*}
Hence,
\[\P\left(\|Ay\|^2\leq 1-\epsilon\right) \leq  \exp \left(d \ln\left(1 - \frac{\ell}{d} + 3\frac{\ell^2}{d^2}\right) + \ell ( 1-\epsilon)\right) \leq \exp \left( 3 \frac{\ell^2}{d} - \ell \epsilon\right)
\;.\] 

Choosing $\ell = d\epsilon/6$, we obtain that
\begin{equation*}
	\P(\|Ay\|^2 \leq 1- \epsilon) \leq \exp(-\frac{d\epsilon^2}{12}) \; .
\end{equation*}

If $d \geq d_0 \geq 12 \ln(3n^2/\delta)/\epsilon^2$, then we obtain 
\begin{equation*}
	\P(\|Ay\|^2 \leq 1- \epsilon) \leq \delta/(3n^2) \; .
\end{equation*}
Hence, from a union bound over the at most $n^2$ possible pairs $x_i, x_j$, we obtain that 
\begin{equation}\label{eq:lb_sparse}
	\P\left( \bigcup_{1\leq i<j\leq n} \left\{ \big\|A(x_i-x_j)\big\|^2 \leq (1- \epsilon)\|x_i-x_j\|^2 \right\}  \right) \leq \delta/3\;.
\end{equation}

We conclude from the upper bound (\ref{eq:ub_sparse}) and the lower bound (\ref{eq:lb_sparse}) that if $q \geq \max_{i\neq j}\frac{\|x_i - x_j\|^2_{\infty}}{\|x_i - x_j\|_2^2}$ and if $d \geq d_0(n,\delta, \epsilon)$, the $\epsilon$-quasi isometry property (\ref{eq:JLprop}) holds with probability at least $1-\delta$, that is
\[\P\left( \bigcup_{1\leq i<j\leq n} \left\{ \big\|A(x_i-x_j)\big\|^2 \notin \left[ (1-\epsilon)\|x_i-x_j\|^2  ,  (1+\epsilon)\|x_i-x_j\|^2  \right] \right\}  \right) \leq 2\delta/3 + \delta/3 \leq \delta\;.\]

\subsection{Proof of Theorem \ref{th:optimality}}
Let $S = \{k \in [p]:~ y_k \neq 0\}$, so that $|S| = s$, and let $\sigma_k = \sqrt{s}\,y_k \in \{-1,1\}$ for $k\in S$. For every $i \in [d]$,
\[
	(Ay)_i = \sum_{k\in S} A_{ik}y_k = \frac{1}{\sqrt{dqs}} \sum_{k\in S} \zeta_{ik} U_{ik}\sigma_k \;.
\]
Let $N_i = \sum_{k \in S}\zeta_{ik} \sim \mathrm{Bin}(s,q)$ and $R_i = \sum_{k\in S} \zeta_{ik} U_{ik}\sigma_k$. Conditionally on $\zeta$, the variables $(U_{ik}\sigma_k)_{k \in S:\, \zeta_{ik} = 1}$ are iid Rademacher variables, so that $R_i$ is a sum of $N_i$ iid Rademacher variables. Writing 
\[
	m := dqs\;, \qquad \lambda := qs = \E[N_i] \;, \qquad T := \sum_{i=1}^d R_i^2\;,
\]
we obtain that $\|Ay\|^2 = T/m$, where the variables $(R_i^2)_{1\leq i\leq d}$ are iid. Recall that we assume $\lambda \leq 1/240$, $m\epsilon^2\leq 1/2$ and $\epsilon\in(0,1)$. We consider two cases.

\medskip
\noindent
{\bf Case $m \leq 1$.} Since $\epsilon<1$, the event $\{Ay = 0\}$ is included in $\{\|Ay\|^2 \notin [1-\epsilon,1+\epsilon]\}$, and as $q\leq \lambda\leq 1/240$,
\begin{align*}
	\P(Ay = 0) &\geq \P\Big(\forall (i,k) \in [d]\times S, \; \zeta_{ik}=0\Big) 
	= (1-q)^{ds} 
	= e^{ds\ln(1-q)}\\
	&\geq e^{-ds\frac{q}{1-q}} = e^{-\frac{m}{1-q}} \geq e^{-\frac{240}{239}} > \frac{1}{3} \;.
\end{align*}
Hence $\P(\|Ay\|^2 \in [1-\epsilon, 1+\epsilon]) \leq 2/3 \leq 19/20$.

\medskip
\noindent
{\bf Case $m \geq 1$.} We first compute the moments of $R_i^2$. If $R$ is a sum of $N$ iid Rademacher variables, then in the expansion of $\E[R^{2k}|N]$ only the index tuples in which every index appears with an even multiplicity contribute; grouping the tuples according to these multiplicities, one obtains, with the notation $N^{(j)} = N(N-1)\cdots(N-j+1)$ for the falling factorials,
\begin{align*}
	\E[R^2|N] &= N\;,\\
	\E[R^4|N] &= 3N^{(2)} + N\;, \\
	\E[R^8|N] &= 105N^{(4)} + 210N^{(3)} + 63N^{(2)} + N\;.
\end{align*}
Moreover, if $N\sim\mathrm{Bin}(s,q)$, then $\E[N^{(j)}] = s(s-1)\cdots(s-j+1)\,q^j \leq \lambda^j$. Hence $\E[R^2] = \lambda$, and the variance $v$ of $R^2$ satisfies
\[
	v := \Var[R^2] = \E[R^4] - \lambda^2 = \lambda + 3s(s-1)q^2 - \lambda^2 = \lambda + \lambda^2\left(2 - \frac{3}{s}\right) \geq \lambda(1-\lambda) \geq \frac{239}{240}\lambda \;,
\]
while, since $|a-b|^4\leq a^4+b^4$ for all $a,b\geq 0$ and since $\lambda\leq 1/240$, the fourth centered moment $\kappa$ of $R^2$ satisfies
\[
	\kappa := \E\big[(R^2-\lambda)^4\big] \leq \E[R^8] + \lambda^4 \leq \lambda + 63\lambda^2 + 210\lambda^3 + 106\lambda^4 \leq 1.3\,\lambda\;.
\]
Now, $T - m = \sum_{i=1}^d (R_i^2 - \lambda)$ is a sum of $d$ iid centered variables with variance $v$ and fourth moment $\kappa$. Hence, $D := \Var[T] = dv \geq \frac{239}{240}\,m \geq \frac{239}{240}$ and
\[
	\E\big[(T-m)^4\big] = d\kappa + 3d(d-1)v^2 \leq 3D^2 + 1.3\, d\lambda \leq 3D^2 + 1.3\times\frac{240}{239}\, D \leq 3D^2 + 1.31\, D \;.
\]
Let $Z = (T-m)^2$, so that $\E[Z] = D$ and $\E[Z^2] \leq 3D^2 + 1.31D$. On the event $\{\|Ay\|^2 \in [1-\epsilon, 1+\epsilon]\}$, we have $|T - m| \leq \epsilon m$, hence 
\[
	Z \leq \epsilon^2 m^2 \leq \frac{m}{2} \leq \frac{120}{239}\, D \;.
\]
The Paley-Zygmund inequality states that for any non-negative random variable $Z$ with finite variance and any $\theta\in(0,1)$, $\P(Z > \theta\E[Z]) \geq (1-\theta)^2\,\E[Z]^2/\E[Z^2]$. Applying it with $\theta = 120/239$, we obtain
\begin{align*}
	\P\big(\|Ay\|^2 \notin [1-\epsilon, 1+\epsilon]\big) &\geq \P\left(Z > \frac{120}{239}\, D\right) \geq \frac{\left(\frac{119}{239}\right)^2 D^2}{3D^2 + 1.31\, D} = \frac{\left(\frac{119}{239}\right)^2}{3 + \frac{1.31}{D}} \\
	&\geq \frac{\left(\frac{119}{239}\right)^2}{3 + 1.31\times\frac{240}{239}} \geq 0.057 > \frac{1}{20}\;,
\end{align*}
which concludes the proof of Theorem \ref{th:optimality}.

\end{document}